\documentclass[english,twoside,a4paper,reqno,11pt]{amsart}
\usepackage{amsfonts, amsbsy, amsmath, amsthm, amssymb, latexsym, verbatim, enumerate}
\usepackage{mathrsfs}
\usepackage[top=30mm,right=30mm,bottom=30mm,left=30mm]{geometry}

\usepackage{bm}
\usepackage[pdftex]{color,graphicx}

\makeatletter
\newcommand{\imod}[1]{\allowbreak\mkern4mu({\operator@font mod}\,\,#1)}
\makeatother

\usepackage[T1]{fontenc}
\usepackage[latin9]{inputenc}
\usepackage{textcomp}
\usepackage{amsmath}
\usepackage{amsthm}
\usepackage{amssymb}

\makeatletter

\theoremstyle{plain}
\newtheorem{thm}{\protect\theoremname}[section]
  \theoremstyle{definition}
  \newtheorem{defn}[thm]{\protect\definitionname}
  \theoremstyle{plain}
  \newtheorem{cor}[thm]{\protect\corollaryname}
  \theoremstyle{plain}
  \newtheorem{lem}[thm]{\protect\lemmaname}
  \theoremstyle{plain}
  
  \theoremstyle{remark}
  \newtheorem*{rem*}{\protect\remarkname}
  \theoremstyle{remark}
  \newtheorem{rem}[thm]{\protect\remarkname}

\usepackage{youngtab} 
\usepackage{ytableau} 
\usepackage{mathrsfs}
\usepackage{pdflscape}
\date{}
\usepackage{changepage}

\usepackage{amsthm}

\theoremstyle{plain}
\newtheorem{mainthm}{Theorem}

\usepackage{amssymb}

\usepackage{enumitem}
\usepackage{longtable}
\usepackage{afterpage}

\DeclareMathOperator{\pw}{pw}
\newcommand{\op}{\text{o}_{p}}

\makeatother

\usepackage{babel}
  \providecommand{\corollaryname}{Corollary}
  \providecommand{\definitionname}{Definition}
  \providecommand{\lemmaname}{Lemma}
  \providecommand{\propositionname}{Proposition}
  \providecommand{\remarkname}{Remark}
\providecommand{\theoremname}{Theorem}

\begin{document}

\author{Alexander J. Malcolm}

\address{Alexander J. Malcolm, School of Mathematics, University of Bristol, Bristol, BS8 1TW, UK, and the Heilbronn Institute for Mathematical Research, Bristol, UK.}
\email{alex.malcolm@bristol.ac.uk}

\title{The $p$-Width of the Alternating Groups}
\subjclass[2010]{Primary 20D06}
\begin{abstract}
 Let $p$ be a fixed prime. For a finite group generated by elements of order $p$, the $p$-width is defined to be the minimal $k\in\mathbb{N}$ such that any group element
can be written as a product of at most $k$ elements of order $p$. Let $A_{n}$  denote the alternating group of even permutations on $n$ letters. We show that
the $p$-width of  $A_{n}$ $(n\geq p)$ is at
most $3$. This result is sharp, as there are families of alternating groups with $p$-width precisely 3, for each prime $p$.
\end{abstract}

\date{\today}
\maketitle

\section{Introduction}
Let $p$ be a fixed prime and $G$ a finite group generated by elements of order $p$. The $p$-width, denoted $\pw(G)$,  is defined to be the minimal $k\in\mathbb{N}$ such that any element of $G$
can be written as a product of at most $k$ elements of order $p$. If $G$ is a non-abelian finite simple group of order divisible by $p$, then it is clear that $G$ is generated by its order $p$ elements. Furthermore, it was shown by Liebeck and Shalev \cite[Thm. 1.5]{LS 01} that there exists an absolute constant that bounds the $p$-width of any such finite simple group.
In this paper we find a sharp bound for the $p$-width of $A_{n} (n\geq p),$ the alternating group of even permutations on $n$ letters.

Even permutations as the product of two conjugate cycles was first considered by Bertram \cite{Bertram 1972}. In particular, he showed that $\lfloor 3n/4 \rfloor\leq p\leq n$ is the necessary and sufficient condition on $p$ in order that every element of $A_{n} (n\geq 5)$  can be expressed as a product of two $p$-cycles (\cite{Bertram 1972} and Thms. \ref{thm:(Bertram):-Regular even Bertram}, \ref{thm:(Bertram):odd bertram} below). This work was extended by Bertram and Herzog \cite{Bertram Herzog} to consider products of three and four $p$-cycles, and then by Herzog, Kaplan and Lev \cite{Herzog et al} to an arbitrary number of $p$-cycles (under some conditions). Products of conjugacy classes in $A_{n}$ have also been studied extensively by Dvir \cite{Dvir}. In particular, \cite[Thm. 10.2]{Dvir} gives a sufficient condition for three copies of a given conjugacy class to cover $A_{n}$ (see Remark \ref{rem: 1st rem} and Appendix \ref{sec:Alt proof})

In the more general case of non-abelian finite simple groups, the problem of the $2$-width (more commonly known as the involution width) has been solved by this author \cite{AJM 2017}: the involution width of any finite simple group is at most four, and this bound is sharp \cite[Thm. 1]{AJM 2017}. Specifically, the involution width of $A_{n}$  is two if $n\in \{5,6,10,14\} $  and three otherwise \cite[Thm. 2.3]{AJM 2017}. 

The work of this paper addresses the alternating groups in the remaining cases where $p\geq 3$ and is summarised by the following theorem.

\begin{mainthm}
Fix a prime $p$. The $p$-width of $A_{n}$ $(n\geq p)$ is at most
three.\label{thm:my theorem}
\end{mainthm} 
The bound in Theorem \ref{thm:my theorem} is
best possible as for each choice of $p$ there exist collections of
alternating groups of $p$-width three. When $p=2$ all but $A_{5},A_{6},A_{10},A_{14}$
have $2$-width three \cite[2.3]{AJM 2017}, and for $p\geq3$ we produce approximately
$2p/3$ alternating groups of $p$-width three (see Remark \ref{rem:sharpness of the alternating groups}
below). 

The proof of Theorem \ref{thm:my theorem} is of a direct combinatorial nature: we show how permutations in $A_{n}$  decompose as  products of elements of order $p$, based on their cycle type (see Sec. \ref{subsec:overview} for a full overview of the proof). We also extensively use the work of Bertram \cite{Bertram 1972} and therefore adopt most of his notation. Note however that we read our cycles from left to right. For example $(1\,2)(1\,3)=(1\,2\,3)$ not $(1\,3\,2).$

\begin{rem} \label{rem: 1st rem}
	A greatly simplified proof of Theorem \ref{thm:my theorem} can be given using the aforementioned result of Dvir \cite[Thm. 10.2]{Dvir} (see Appendix \ref{sec:Alt proof} for a full explanation). The author was unaware of this result when the work in this paper was completed.
\end{rem} 

The  $p$-width is one of a number of width questions that have
been considered about simple groups in recent literature. For example,
\cite{Ore Conjecture paper} settles the longstanding conjecture of
Ore that the commutator width of any finite non-abelian simple group
$G$ is exactly one. Also, $G$ is generated by its set of squares and
the width in this case is two \cite{lost squares}. More generally,
given any two non-trivial words $w_{1},w_{2}$, if $G$ is of large
enough order then $w_{1}(G)w_{2}(G)\supseteq G\backslash\{1\}$ \cite{Guralnick and Tiep}. 
The $p$-width of the sporadic finite simple groups (as well as some low-rank groups of Lie type) is addressed in the author's PhD thesis \cite{AJM Thesis}. Further work on the $p$-width of the finite simple groups of Lie type will be forthcoming.

\textit{Acknowledgments}. The work in this paper forms part of a PhD thesis, completed under the supervision of Martin Liebeck at Imperial College London, and the author wishes to thank him for his support and guidance throughout. The author is also thankful to EPSRC for their financial support during his PhD.

\section{Preliminary Material}
For the remainder of this paper, $p$  will denote a fixed odd prime, and $A_{n}$ is such that $n\geq	5$ and $n\geq	p$.

\begin{defn}
We call elements of order $p$, $\op$-elements and define
\[
I_{p}(A_{n}):=\{g\in A_{n}\,|\,g^{p}=1\}.
\]
The $p$-width
of a permutation $g\in A_{n}$, denoted $\pw(g)$, is the minimal $k\in\mathbb{N}$
such that $g$ can be written as a product of $k$ elements in $I_{p}(A_{n})$.
\end{defn}
\begin{rem*}
Recall that elements of order $p$ have restricted cycle structure:
$g\in S_{n}$ has prime order if and only if it is the product of
a finite number of $p$-cycles. The inclusion of the identity in $I_{p}(A_{n})$ is simply for convenience:
the methods in Section \ref{sec:Main Sec} may produce trivial
elements depending on the permutations considered in $A_{n}$. Naturally
any trivial elements will then be ignored when considering products
in $I_{p}(A_{n})$ of minimal length.
\end{rem*}

\begin{defn}
\label{def:Initial definitions of An}Let $g\in S_{n}.$ Denote the
support of $g$ (i.e. the letters that are moved by $g$) by the set
$\mu(g)$. Also, denote by $c^{*}(g)$, the number of non-trivial
cycles in the decomposition of $g$. 

Let $g=g_{1}\dots g_{k}$ be a product of $g_{i}\in S_{n}.$ If there
exists $x\in\mu(g_{i})$ such that $x\notin\mu(g_{j})$ for $i\neq j$
then we call $x$ a \textit{free letter} for $g$, with respect to
the decomposition $g=g_{1}\dots g_{k}$.

Let $g=c_{1}\dots c_{k}\in S_{n}$ be an element written as a product
of disjoint cycles and let $I\subsetneq\{1,\dots,k\}$. We call $g_{I}:=\prod_{i\in I}c_{i}$
a \textit{sub-element} of $g$, corresponding to the \textit{sub-collection
of cycles} $\{c_{i}\}_{i\in I}$. Furthermore, if $g_{I}$ is such
a sub-element then we define $g\backslash g_{I}:=\prod_{i\notin I}c_{i}$.
\end{defn}

The two main results of Bertram (\cite{Bertram 1972}) are essential to the proof of Theorem \ref{thm:my theorem} and so we repeat them here for reference:
\begin{thm}
\label{thm:(Bertram):-Regular even Bertram}\cite[Thm. 1 and 2]{Bertram 1972}. Let $g\in A_{n}\setminus\{1\}$ and define
\[
l(g)=\frac{|\mu(g)|+c^{*}(g)}{2}.
\]
 Then there exist two $l$-cycles $A$ and $B$ such that $g=AB$
if and only if $n\geq l\geq l(g).$ Also, $\max_{g\in A_{n}}l(g)=\lfloor\frac{3n}{4}\rfloor$. 
\end{thm}

\begin{thm}
\label{thm:(Bertram):odd bertram}\cite[Thm. 3]{Bertram 1972}. Let
$g\in S_{n}\backslash A_{n}.$ Let $l$ be an integer satisfying 
\[
n-1\geq l\geq\frac{|\mu(g)|+c^{*}(g)-1}{2}.
\]
 Then $g$ may be expressed as a product $AB$ of an $(l+1)$-cycle
$A$ and $l$-cycle $B$.
\end{thm}
\begin{rem}
\label{rem:successively}
 It is important to note the following from the proof of Theorems
\ref{thm:(Bertram):-Regular even Bertram} and \ref{thm:(Bertram):odd bertram}:
let $g\in A_{n}$ and assume that Theorem \ref{thm:(Bertram):-Regular even Bertram}
holds for some $l$. If $|\mu(g)|\geq l$ then in the proof of Theorem
\ref{thm:(Bertram):-Regular even Bertram}, the cycles $A$ and $B$
are chosen such that $\mu(A),\,\mu(B)\subseteq\mu(g).$ If however
$|\mu(g)|<l,$ then the cycles $A$ and $B$ require additional letters
to those in $\mu(g)$. Precisely, we require $l-|\mu(g)|$ new letters.
It will be important to differentiate between these two cases so we
define $AB$ to be a \textit{strong decomposition} or a \textit{weak
decomposition,} when $|\mu(g)|\geq l$ or $|\mu(g)|<l$, respectively.
Similarly, if $g\in S_{n}\backslash A_{n}$ and Theorem \ref{thm:(Bertram):odd bertram}
holds, then we have a strong or weak decomposition when $|\mu(g)|\geq l+1$
or $|\mu(g)|<l+1$, respectively. Lastly, if $g\in A_{n}$ and $g=x_{1}x_{2}x_{3}$
for some $x_{i}\in I_{p}(A_{n})$ such that $\mu(x_{i})\subseteq\mu(g)$,
then we say $g$ has a strong decomposition in $I_{p}(A_{n})^{3}$.

It is also clear from the proof of Theorem \ref{thm:(Bertram):-Regular even Bertram} (\cite{Bertram 1972},
Thm. 1) that should a decomposition into two $l$-cycles be weak, then any additional letters can be written successively in the elements $A$ and $B$. For example suppose that  $\mu(g)=\{a_{1},\dots,a_{n}\}=\{a_{1}, b_{2},\dots,b_{n}\}$ and
\[ 
g =  (a_{1},\dots ,a_{n})\cdot (b_{2},\dots ,b_{n}, a_{1}).
\]
Then if $o_{1},\dots,o_{l-n}$ are additional letters distinct from  $a_{1},\dots ,a_{n}$  we can extend the decomposition to $l$-cycles
\[
g=(o_{1},\dots ,o_{l-n}, a_{1},\dots ,a_{n})\cdot (b_{2},\dots ,b_{n}, a_{1}, o_{l-n},\dots ,o_{1}). 
\]
\end{rem}
 
Theorems \ref{thm:(Bertram):-Regular even Bertram} and \ref{thm:(Bertram):odd bertram}
use the notation of \cite{Bertram 1972} with permutations denoted
by capital letters. However in this work we shall reserve capital
letters for permutations that are known to be $p$-cycles. These $p$-cycles
will predominantly appear through the application of said theorems.
Furthermore, we shall only apply Theorem \ref{thm:(Bertram):-Regular even Bertram}
or Theorem \ref{thm:(Bertram):odd bertram} with $l=p$ or $l=p-1$,
respectively. This allows us to avoid stating the value of $l$ each
time.

\begin{rem}
\label{rem:sharpness of the alternating groups}Note that Theorem
\ref{thm:(Bertram):-Regular even Bertram} implies that Theorem \ref{thm:my theorem}
is in general best possible: as $\max_{g\in A_{n}}l(g)=\lfloor \frac{3n}{4} \rfloor$, if $2p>n$ and  $\frac{3n-3}{4}>p$ then $A_{n}\neq I_{p}(A_{n})^{2}$.
Rearranging gives that for each fixed $p$, $\pw(A_{n})\geq3$ for $2p>n>\frac{4p+3}{3}$.
Furthermore, when $p\geq5$ this collection of alternating groups
is non-empty. Finally, when $p=3$ we can check by hand that $\pw(A_{5})=3$. 

We will also use the following lemma in the latter parts of the proof.
\end{rem}
\begin{lem}
\label{lem:lemma}\cite[Lem. 2]{Bertram 1972}. Let $a$ and $b$
be single cycles in $S_{n}$. If $|\mu(a)\cap\mu(b)|=m\geq1$ then
$c^{*}(ab)\leq m.$ 
\end{lem}

\section{Proof of main result} \label{sec:Main Sec}
\subsection{Overview of the proof of Theorem \ref{thm:my theorem}} \label{subsec:overview}

Let $\sigma\in A_{n}$ and consider the disjoint cycle decomposition
\begin{equation}
\sigma=\prod_{i=1}^{k_{a}}a_{i}\prod_{j=1}^{k_{c}}c_{j}\prod_{s=1}^{k_{b}}b_{s}\prod_{t=1}^{k_{d}}d_{t},\label{eq:disjoint decomp of our initial element}
\end{equation}
 where $a_{i}$, $c_{j}$ are cycles of odd length such that $|\mu(c_{j})|<p\leq|\mu(a_{i})|$
and the collections $\{b_{s}\}$ and $\{d_{t}\}$ consist of even
length cycles. These collections are labeled such that $k_{b},k_{d}$
are even (we want $\prod_{s=1}^{k_{b}}b_{s},\,\prod_{t=1}^{k_{d}}d_{t}\in A_{n}$)
: let $\{b_{s}\}$ denote the cycles of even length greater than $p$,
plus a single cycle of even length less than $p$ if there an odd
number of these in $\sigma$. Finally let $\{d_{t}\}$ denote the
set of remaining cycles of even length less than $p$.

We shall prove that $\pw(\sigma)\leq3$ in the following steps:
\begin{enumerate}
\item We show that for any cycle $a_{i}$, $\pw(a_{i})\leq3$. Furthermore,
$a_{i}$ has a strong decomposition in $I_{p}(A_{n})^{3}$.
\item We show that $\pw(b_{s}b_{s+1})\leq3$ for any pair $b_{s}b_{s+1}$
such that $b_{s}$ has length greater than $p$. Furthermore, $b_{s}b_{s+1}$
has a strong decomposition in $I_{p}(A_{n})^{3}$. 
\item Consider the sub-element $g=\prod_{j=1}^{k_{c}}c_{j}\prod_{t=1}^{k_{d}}d_{t}$
of $\sigma$, and assume that $(|\mu(g)|+c^{*}(g))/2\leq p$. If $\sigma=g$
then we show that $\pw(g)\leq2$. If $\sigma\neq g$ then we show
that $\pw(gh)\leq3$ for any non-trivial sub-element $h$ of $\prod_{i=1}^{k_{a}}a_{i}\prod_{j=1}^{k_{c}}c_{j}$.
In the latter case, $gh$ has a strong decomposition in $I_{p}(A_{n})^{3}$.
Note that Part (3) relies on the work of Parts (1) and (2) above.
\item Consider the sub-element $g=\prod_{s=1}^{k_{c}}c_{s}\prod_{t=1}^{k_{d}}d_{t}$
of $\sigma$, and assume that $(|\mu(g)|+c^{*}(g))/2>p$. We show
that we can apply either strong Theorem \ref{thm:(Bertram):-Regular even Bertram}
or strong Theorem \ref{thm:(Bertram):odd bertram} repeatedly to sub-elements
of $g$. We construct $\op$-elements during this process to show $\pw(g)\leq3$
and that $g$ has a strong decomposition in $I_{p}(A_{n})^{3}$.
\item The result follows from (1)-(4) as all the $a_{i},\,b_{s},\,c_{j},\,d_{t}$
commute and we group disjoint $\op$-elements into 3 families as required.
\end{enumerate}

\subsection{Cycles in $\{a_{i}\}$\label{sec:ai cycles}}

In this section we consider cycles in $\{a_{i}\}_{i=1}^{k_{a}}$ and
prove the following Lemma:
\begin{lem}
\label{lem:a cycles lemma}Let $a\in A_{n}$ be single cycle of odd
length of at least $p$. Then $\pw(a)\leq3$ and $a$ has a strong
decomposition in $I_{p}(A_{n})^{3}$. Furthermore, the strong decomposition
of $a$ contains free letters in the following amounts:

\begin{enumerate}
\item if $|\mu(a)|\geq5p-4$ - at least $2(p-2)$ free letters;
\item if $2p\leq|\mu(a)|\leq3p-1$ - at least $p-1$ free letters;
\item if $4p-2\leq|\mu(a)|\leq5p-5$ - at least $2(p-1)$ free letters;
\item if $3p\leq|\mu(a)|\leq4p-3$ - at least $2(p-1)$ free letters;
\item if $p\leq|\mu(a)|\leq2p-1$ - no minimum number of free letters.
\end{enumerate}
\end{lem}
\begin{proof}
Without loss of generality let $a=(1\dots n)$, where $n$ is odd.
Naturally if $n=p$ then the Lemma is immediate and hence we assume
that $n>p$. To allow for easy referencing in later material, we break
the proof into the cases $(1)-(5)$ as specified in the statement
of the Lemma. 

\textbf{Case (1): }Assume that $n\geq5p-4.$ For illustrative purposes
consider the following example when $p=5$ and $n=47:$ 
\begin{eqnarray*}
(1\,2\,3\dots47) & = & (1\dots5)(13\dots16\,47)(20\dots23\,46)(27\dots30\,45)(34\dots37\,44)\\
 & \times & (1\,6\dots9)(17\dots20\,47)(24\dots27\,46)(31\dots34\,45)\\
 & \times & (1\,10\,11\,12\,13)\\
 & \times & (38\dots44).
\end{eqnarray*}
 More generally we decompose $a$ as follows: let $m:=n-(3p-2)$ and
choose $l$ such that $p-1\leq l<3p-2$ and $m=c\cdot(2p-2)+(p-1)+l$,
where $c\in\mathbb{N}$. Note this is possible as $n\geq5p-4$. We
then have that
\begin{eqnarray*}
(1\dots n) & = & (1\dots p)\prod_{i=0}^{c}((3p-2)+i(2p-3),\dots,(4p-4)+i(2p-3),n-i))\\
 & \times & (1\,p+1\dots2p-1)\\
 &  & \prod_{i=1}^{c}((4p-3)+(i-1)(2p-3),\dots,(5p-5)+(i-1)(2p-3),n-i+1))\\
 & \times & (1\,2p\dots3p-2)\\
 & \times & (n-c-l,\dots,n-c).
\end{eqnarray*}

Denote the collections of cycles on the 1st, 2nd and 3rd, and 4th
rows by $x_{1}=s_{1}\dots s_{c+2},$  $x_{2}=t_{1}\dots t_{c+1}$ and
$x_{3}=u_{1}$ respectively, and denote the final cycle (of length
$l+1$) by $r$. In summary 
\[
a=x_{1}x_{2}x_{3}\cdot r,
\]
where $x_{1},x_{2},x_{3}\in I_{p}(A_{n})$ and $r$ is the ``remainder''
of this construction. 

Note that $\mu(r)\cap\mu(x_{2})=\mu(r)\cap\mu(x_{3})=\emptyset$ and
$\mu(r)\cap\mu(x_{1})=\{n-c\}$. Hence if $|\mu(r)|=p$ it follows
that $a=x_{1}\cdot x_{2}\cdot(x_{3}r)\in I_{p}(A_{n})^{3}$ and that this
is a strong decomposition into $\op$-elements.

Now assume that $p<|\mu(r)|<2p$. It follows that 
\[
\frac{|\mu(r)|+c^{*}(r)}{2}\leq p
\]
 and we apply strong Theorem \ref{thm:(Bertram):-Regular even Bertram}
to write $r$ as a product of two $p$-cycles $r=E_{1}E_{2}$. Thus
there exists a strong decomposition
\[
a=x_{1}\cdot x_{2}E_{1}\cdot x_{3}E_{2}\in I_{p}(A_{n})^{3}.
\]
 The remaining case to consider is where $2p\leq|\mu(r)|\leq3p-2$:
we first split $r$ into 2 cycles, $r=e_{1}E_{2}$ where $|\mu(E_{2})|=p$,
$|\mu(e_{1})|=l-p+2$ and $\mu(e_{1}),\mu(E_{2})\subset\mu(r)$. These
cycles can be chosen such that $n-(c-2)\in\mu(E_{2})$ and $n-(c-2)\notin\mu(e_{1})$. 

\textit{Remark: }This is a method we apply throughout this paper
and we refer to it as \textit{splitting off a $p$-cycle}. Specifically,
if $o=(o_{1},\dots,o_{l})\in S_{n}$ is a single cycle of length $l>p$,
then we split off a $p$-cycle by writing
\[
(o_{1},\dots,o_{l})=(o_{1},\dots,o_{l-p+1})(o_{1},o_{l-p+2},\dots,o_{l})=:o'\cdot P.
\]

Clearly $\mu(o'),\mu(P)\subset\mu(o)$ and |$\mu(o')\cap\mu(P)|=1$,
where the letter in this intersection can be chosen from any in $\mu(o)$.
Note that $|\mu(o')|=l-p+1$.

Now as $p+1\leq|\mu(e_{1})|\leq2p-1$ we use strong Theorem \ref{thm:(Bertram):-Regular even Bertram}
 to write $e_{1}$ as a product of two $p$-cycles, say $e_{1}=F_{1}F_{2}$.
Clearly $\mu(F_{i})\cap(\mu(x_{1})\cup\mu(x_{2})\cup\mu(x_{3}))=\emptyset$.
Hence cycles commute where necessary to yield a strong decomposition
\[
a=x_{1}x_{2}x_{3}r=x_{1}F_{1}\cdot x_{2}F_{2}\cdot x_{3}E_{2}\in I_{p}(A_{n})^{3}.
\]
 In summary, if $|\mu(a)|\geq5p-4$ then there exists a strong decomposition
$a=y_{1}y_{2}y_{3}\in I_{p}(A_{n})^{3}$. Furthermore, it is clear from the above construction of $y_{i}\in I_{p}(A_{n})$, that there exist at
least $p-2$ free letters in each of the sets $\mu(y_{1})\backslash(\mu(y_{2})\cup\mu(y_{3}))$
and $\mu(y_{2})\backslash(\mu(y_{1})\cup\mu(y_{3}))$. This completes
the proof of Lemma \ref{lem:a cycles lemma} Case (1).

\textbf{Case (2): }Assume that $2p\leq n\leq3p-1$. This is addressed
exactly as in case (1) by letting $a=r$: we write a strong decomposition
$a=(1\dots n)=F_{1}F_{2}E_{2}\in I_{p}(A_{n})^{3}$ and note that $\mu(E_{2})\cap(\mu(F_{1})\cup\mu(F_{2}))$
is a singleton set. Hence the decomposition of $a$ contains $p-1$
free letters in the cycle $E_{2}$.

\textbf{Case (3):} Assume that $4p-3<n<5p-4$. Firstly note that 
\begin{eqnarray*}
(1\dots n) & = & (1\dots p)(1\,p+1\dots2p-1)\\
 & \times & (2p\dots n-p+2)(2p,\,n-p+3\dots n-1,\,n,\,1).
\end{eqnarray*}
 Let $g:=(2p\dots n-p+2)$. By assumption, $p<|\mu(g)|<2p-1$ and
thus we can apply strong Theorem \ref{thm:(Bertram):-Regular even Bertram}
to write $g$ as a product of two $p$-cycles, $g=AB$. Furthermore
we can assume that $B$ fixes $2p$. Thus labeling
\begin{eqnarray*}
y_{1} & = & (1\dots p)\\
y_{2} & = & (1\,p+1\dots2p-1)A\\
y_{3} & = & B(2p\,n-p+3\dots n-1\,n\,1)
\end{eqnarray*}
 yields a strong decomposition $(1\dots n)=y_{1}y_{2}y_{3}\in I_{p}(A_{n})^{3}$.
Note that $y_{1}$ and $y_{2}$ both contain $p-1$ distinct free
letters, completing the proof of Lemma \ref{lem:a cycles lemma} Case
(3).

\textbf{Case (4):} Assume that $3p\leq n\leq4p-3$ and first write
\[
(1\dots n)=(1\dots p)(1\,p+1\dots2p-1)(1\,2p\dots n).
\]
Labeling the final cycle by $g$, it follows that $p+2\leq|\mu(g)|\leq2p-1$
and we again apply strong Theorem \ref{thm:(Bertram):-Regular even Bertram}
to write $g=AB$, a product of two $p$-cycles. We assume without
loss of generality that $1\notin\mu(A)$ and thus 
\[
(1\dots n)=(1\dots p)\cdot(1\,p+1\dots2p-1)A\cdot B\in I_{p}(A_{n})^{3},
\]
is a strong decomposition satisfying Lemma \ref{lem:a cycles lemma}
Case (4). 

\textbf{Case (5):} Assume that $p<n\leq2p-1.$ Here a simple application
of strong Theorem \ref{thm:(Bertram):-Regular even Bertram} will
suffice and we can write $a=AB$ the product of two $p$-cycles. 

This completes the proof of Lemma \ref{lem:a cycles lemma}. 
\end{proof}

\subsection{Pairs of cycles in $\{b_{i}\}$\label{sec:Pairs-of-cycles}}

In this section we show that $\pw(bb')\leq3$ for
any pair of cycles $b,b'\in\{b_{s}\}_{s=1}^{k_{b}}$ (see Lemma
\ref{lem:b cycles lemma} below). Recall that these cycles have even
length and without loss of generality, $|\mu(b)|>p$. In general
we aim to write $bb'$ as a product of odd length cycles and
then apply the methods in Section \ref{sec:ai cycles}. Furthermore, we follow a similar procedure in considering various cases
dependent on the size of $\mu(bb')$.
\begin{lem}
\label{lem:b cycles lemma}Let $b_{1}b_{2}\in A_{n}$ be a pair of
cycles of even lengths such that $|\mu(b_{1})|>p$. Then $\pw(b_{1}b_{2})\leq3$
and $b_{1}b_{2}$ has a strong decomposition in $I_{p}(A_{n})^{3}$. Furthermore,
the strong decomposition of $b_{1}b_{2}$ contains free letters in
the following amounts:
\end{lem}
\begin{enumerate}
\item if $|\mu(b_{1})|\geq2p+2$ and $|\mu(b_{2})|>p$ - at least $p-2$
free letters;
\item if $p<|\mu(b_{i})|\leq2p$ - at least $p-1$ free letters;
\item if $|\mu(b_{1})|\geq2p$ and $|\mu(b_{2})|<p$ - at least $p-2$ free
letters; 
\item if $p<|\mu(b_{1})|\leq2p-2$ and $|\mu(b_{2})|<p$ - no minimum number
of free letters.
\end{enumerate}
\begin{proof}
\textbf{Case (1): }First consider a pair of cycles $b_{1}b_{2}$ such
that $|\mu(b_{1})|\geq2p+2$ and $|\mu(b_{2})|>p$. 

Without loss of generality let $b_{1}=(1\dots n)$ and $b_{2}=(\bar{1}\dots\bar{n})$.
These can be written as a product of a 2-cycle and an odd-cycle with
length at least $p$ as follows:
\[
b_{1}=(1\,3\dots n)(2\,3)
\]
and 
\[
b_{2}=(\bar{1}\,\bar{3}\dots\bar{n})(\bar{2}\,\bar{3}).
\]
 By Lemma \ref{lem:a cycles lemma}, the odd-length cycles can be
written as strong products of at most $3$ $\op$-elements (denote
these $x_{i}$ and $\overline{x}_{i}$) to give the decomposition
\[
b_{1}b_{2}=x_{1}x_{2}x_{3}\bar{x}_{1}\bar{x}_{2}\bar{x}_{3}(2\,3)(\bar{2}\,\bar{3}).
\]
 As the $x_{i}$ and $\bar{x}_{i}$ contain distinct letters, they
pairwise commute to give 
\begin{equation}
b_{1}b_{2}=(x_{1}\bar{x}_{1})(x_{2}\bar{x}_{2})(x_{3}\bar{x}_{3})(2\,3)(\bar{2}\,\bar{3}),\label{eq:decomp}
\end{equation}
 where $x_{i}\overline{x}_{i}\in I_{p}(A_{n})$. The element $(2\,3)(\bar{2}\,\bar{3})$
can be easily written as a product of two $p$-cycles but we require
$p-4$ new letters, call these $o_{i}$. e.g 
\begin{equation}
(2\,3)(\bar{2}\,\bar{3})=(o_{1}\dots o_{p-4}\bar{3}\,3\,\bar{2}\,2)(3\,\bar{2}\,2\,\bar{3}\,o_{p-4}\dots o_{1}).\label{eq:(23)(23) construction}
\end{equation}
 The statement of Lemma \ref{lem:a cycles lemma} (and its proof)
show where to find these free letters. In particular, as $|\mu(b_{1})|\geq2p+2$,
the $\op$-element $x_{1}$ in (\ref{eq:decomp}) contains at least
$p-1$ letters not contained in any other $\mu(x_{i})$ or $\mu(\bar{x}_{i})$.
Excluding possibly the letter $3$, this leaves $p-2$ free letters
to use in the construction (\ref{eq:(23)(23) construction}). In summary
we can find appropriate $o_{1},\dots,o_{p-4}\in\mu(x_{1})\backslash(\mu(x_{2})\cup\mu(\bar{x}_{2})\cup\mu(x_{3})\cup\mu(\bar{x}_{3})\cup\{3\})$.
Commuting cycles where necessary yields
\[
b_{1}b_{2}=(x_{1}\bar{x}_{1})(x_{2}\bar{x}_{2}(o_{1}\dots o_{p-4}\bar{3}\,3\,\bar{2}\,2))(x_{3}\bar{x}_{3}(3\,\bar{2}\,2\,\bar{3}\,o_{p-4}\dots o_{1})).
\]
Hence $\pw(b_{1}b_{2})\leq3$, and this is clearly a strong decomposition
in $I_{p}(A_{n})^{3}$. 

\textit{Remark:} To complete the proof of Lemma \ref{lem:b cycles lemma}
Case (1) we need to show that free letters exist in the strong decomposition
of $b_{1}b_{2}$. This is clear by Lemma \ref{lem:a cycles lemma}
Cases (1)-(4) if either $|\mu(b_{1})|>3p-1$ or $|\mu(b_{2})|>2p$:
there are at least $p-2$ free letters in $x_{2}$ or $\overline{x}_{1}$ respectively.
If $|\mu(b_{1})|\leq3p-1$ and $|\mu(b_{2})|\leq2p$ then we claim
that there also exists a strong decomposition of $b_{1}b_{2}\in I_{p}(A_{n})^{3}$
that contains free letters, but it is not immediate from Lemma \ref{lem:a cycles lemma}.
This claim is addressed by Case (1') below, where we take a different
approach to re-prove that $b_{1}b_{2}\in I_{p}(A_{n})^{3}$ while making
it clear where the free letters are available.

\textbf{Case (1'): }Consider $b_{1}b_{2}$ where $2p+2\leq|\mu(b_{1})|\leq3p-1$
and $p+1\leq|\mu(b_{2})|\leq2p$. Firstly split off a $p$-cycle from
both $b_{1}$ and $b_{2}$ (see Remark in the proof of Lemma
\ref{lem:a cycles lemma} Case (1)) to give $b_{1}b_{2}=P_{1}P_{2}b'_{1}b'_{2}$
such that $|\mu(b_{i})\cap\mu(P_{i})|=1$. Clearly $|\mu(b'_{1}b'_{2})|\geq p$
and if $|\mu(b'_{1}b'_{2})|+c^{*}(b'_{1}b'_{2})\leq2p$ then we apply
strong Theorem \ref{thm:(Bertram):-Regular even Bertram} to give
$b'_{1}b'_{2}=AB\in I_{p}(A_{n})^{2}$. It then follows that $b_{1}b_{2}=P_{1}P_{2}\cdot A\cdot B\in I_{p}(A_{n})^{3}$
is a strong decomposition and there exist $p-1$ free letters in the
cycle $P_{1}$. If instead $|\mu(b'_{1}b'_{2})|+c^{*}(b'_{1}b'_{2})>2p$,
the result follows in a similar manner after splitting of an additional
$p$-cycle from $b'_{1}$ (this is always possible as $2p+2\leq|\mu(b_{1})|$)
and then also from $b'_{2}$ if necessary. In particular if $b_{i}=P_{i}P'_{i}b''_{i}$
then $|\mu(b''_{1}b''_{2})|+c^{*}(b''_{1}b''_{2})\leq p+5$ and hence
we can apply strong Theorem \ref{thm:(Bertram):-Regular even Bertram}
if $p\geq5$. The case $p=3$ can be checked by hand. In summary it
is clear in all cases that $\pw(b_{1}b_{2})\leq3$ and $P_{1}$ contains
$p-1$ free letters, completing the proof of Lemma \ref{lem:b cycles lemma}
Case (1).

\textbf{Case (2): }Consider $b_{1}b_{2}$ where both cycles have length
$p<|\mu(b_{i})|\leq2p$. We cannot immediately use Theorem \ref{thm:(Bertram):-Regular even Bertram}
so first split off a $p$-cycle from each element, rewriting 
\[
b_{1}b_{2}=b_{1}^{'}b_{2}^{'}P_{1}P_{2},
\]
 where $|\mu(b_{i})\cap\mu(P_{i})|=1$ and $4\leq|\mu(b'_{1}b'_{2})|\leq2p+2$.
If $|\mu(b'_{1}b'_{2})|\leq2p-2$ then apply Theorem \ref{thm:(Bertram):-Regular even Bertram}
to write $b'_{1}b'_{2}$ as a product of two $p$-cycles $AB$. This
may require extra letters if $|\mu(b'_{1}b'_{2})|<p$ but these can
be taken from $\mu(P_{1}P_{2})\backslash\mu(b_{1}^{'}b_{2}^{'})$
if necessary. Consequently, $b_{1}b_{2}=A\cdot B\cdot(P_{1}P_{2})\in I_{p}(A_{n})^{3}$
and $\mu(P_{i}),\mu(A),\mu(B)\subseteq\mu(b_{1}b_{2})$.  

This leaves the cases where $|\mu(b'_{1}b'_{2})|=2p$ or $2p+2$.
But recall that $p<|\mu(b_{i})|\leq2p$ and so $|\mu(b'_{i})|=|\mu(b_{i})|-p+1\leq p+1$.
It follows without loss of generality that $|\mu(b'_{2})|=p+1$ and
hence it is possible split off a further $p$-cycle $P_{2}'$ such
that $b'_{2}=b''_{2}P'_{2}$, $\mu(P'_{2})\cap\mu(P_{2})=\emptyset$
and $b''_{2}=(a\,c)$ for some $c\in\mu(P'_{2}).$ Noting that $|\mu(b'_{1}b''_{2})|=p+1$
or $p+3$, we apply Theorem \ref{thm:(Bertram):-Regular even Bertram}
for a strong decomposition $b'_{1}b''_{2}=AB$ and thus $b_{1}b_{2}=A\cdot B\cdot(P_{1}P'_{2}P_{2})\in I_{p}(A_{n})^{3}$.  It is clear in each decomposition that $P_{2}$ contains $p-1$ free letters.

Cases (1) and (2) give a complete treatment of cycle pairs $b_{1}b_{2}$
such that $|\mu(b_{i})|>p$. It therefore remains to consider a pair
of even length cycles $b_{1}b_{2}$ such that $|\mu(b_{2})|<p<|\mu(b_{1})|$. 

\textbf{Case (3): }Assume that $2p\leq|\mu(b_{1})|$.  Without loss of generality let
$b_{1}=(1\dots n)=(1\dots n-1)(1\,n)$. The cycle $(1\dots n-1)$ has
odd length at least $2p-1$ and so by Lemma \ref{lem:a cycles lemma}
there exists a strong decomposition $(1\dots n-1)=y_{1}y_{2}y_{3}$
for some $y_{1},y_{2},y_{3}\in I_{p}(A_{n})$. Furthermore, $y_{1}$ contains
$p-1$ free letters. Now note that $|\mu((1\,n)b_{2})|\leq p+1$ and hence
we can apply Theorem \ref{thm:(Bertram):-Regular even Bertram}. 

If $|\mu((1\,n)b_{2})|=p+1$ then this produces a strong decomposition
$(1\,n)b_{2}=AB$. If instead $|\mu((1\,n)b_{2})|<p+1$ then Theorem \ref{thm:(Bertram):-Regular even Bertram}
yields a weak decomposition requiring at most $p-4$ letters. But
these can be taken easily from $y_{1}$ and thus we write $b_{1}b_{2}=y_{1}y_{2}y_{3}(1\,n)b_{2}=y_{1}\cdot y_{2}A\cdot y_{3}B\in I_{p}(A_{n})^{3}$,
a strong decomposition. 

As in Case (1), it is clear that the conclusion of Lemma \ref{lem:b cycles lemma}
Case (3) holds if $|\mu(b_{1})|>3p-1$. That $p-2$ free letters exist
in the strong decomposition of $b_{1}b_{2}$ when $|\mu(b_{1})|\leq3p-1$ follows
in an almost identical manner to Case (1').

\textbf{Case (4): }Assume that $p+1\leq|\mu(b_{1})|\leq2p-2.$ We first
split off a $p$-cycle $b_{1}=Pb_{1}'$ to leave $2\leq|\mu(b_{1}')|\leq p-1$. It follows that $4\leq\mu(b_{1}'b_{2})\leq2p-2$
and hence we apply Theorem \ref{thm:(Bertram):-Regular even Bertram}
to get  $b_{1}'b_{2}=AB$. If this is a weak $p$-cycle
decomposition then we choose any necessary free letters (at most
$p-4$\textbf{ }needed) from the $p-1$ available in the cycle $P$.
Clearly then $b_{1}b_{2}=P\cdot A\cdot B$ is a strong decomposition in $I_{p}(A_{n})^{3}.$ 

This completes the proof of Lemma\textbf{ }\ref{lem:b cycles lemma}
\end{proof}

\subsection{Cycles in $\{c_{i}\}$ and pairs of cycles in $\{d_{i}\}$; initial
considerations\label{sec:  c and d cycles, inital considerations}}

In this section we consider the collections of cycles $\{c_{i}\}$
and $\{d_{i}\}$ in the decomposition (\ref{eq:disjoint decomp of our initial element})
of $\sigma$ and prove Lemma \ref{lem:first cd lemma}. Recall that
these cycles have length less than $p$, and are of odd and even lengths
respectively. We take a slightly different approach from the previous
two sections and consider the collections $\prod_{j=1}^{k_{c}}c_{j}$
and $\prod_{t=1}^{k_{d}}d_{t}$ as a whole, rather than the individual
cycles. Further recall that $k_{d}$ is assumed even as this guarantees
that $g=\prod_{j=1}^{k_{c}}c_{j}\prod_{t=1}^{k_{d}}d_{t}\in A_{n}.$
 In this section we assume additionally that $g$ satisfies 
\[
(|\mu(g)|+c^{*}(g))/2\leq p.
\]
 The case where this does not hold is considered in Section \ref{sec:Can't apply Bertram}.
\begin{lem}
\label{lem:first cd lemma}Let $g=\prod_{j=1}^{k_{c}}c_{j}\prod_{t=1}^{k_{d}}d_{t}$
be the sub-element of $\sigma$ as described in (\ref{eq:disjoint decomp of our initial element})
and assume that $(|\mu(g)|+c^{*}(g))/2\leq p$. If $\sigma=g$ or
$|\mu(g)|\geq p$ then $\pw(g)\leq2$. Suppose instead that $|\mu(g)|<p$,
and $h$ is a non-trivial sub-element of $\prod_{i=1}^{k_{a}}a_{i}\prod_{j=1}^{k_{c}}c_{j}$
of one of the following forms:

\begin{enumerate}
\item a single cycle $h=a$ of odd length such that $p\leq|\mu(a)|\leq2p-1$;
\item a single cycle $h=a$ of odd length such that $2p\leq|\mu(a)|$;
\item a pair of even length cycles $h=b_{1}b_{2}$ such that $|\mu(b_{1})|\geq2p+2$
and $|\mu(b_{2})|>p$;
\item a pair of even length cycles $h=bd$ such that $|\mu(d)|<p<|\mu(b)|$;
\item a pair of even length cycles $h=b_{1}b_{2}$ such that $p<|\mu(b_{i})|\leq2p$.
\end{enumerate}
Then $\pw(gh)\leq3$ and there exists a strong decomposition of $gh\in I_{p}(A_{n})^{3}$. 
\end{lem}
\begin{proof}
Clearly if $|\mu(g)|\geq p$ then by Theorem \ref{thm:(Bertram):-Regular even Bertram},
we can write $g=AB$ a strong decomposition of $p$-cycles and Lemma
\ref{lem:first cd lemma} holds. So we assume instead that $|\mu(g)|<p$
and so a weak application of Theorem \ref{thm:(Bertram):-Regular even Bertram}
requires $p-|\mu(g)|$ extra letters, where $1\leq p-|\mu(g)|\leq p-3$.
If additionally $g=\sigma$, it is  straightforward to see that
$\pw(g)\leq2$: recall that $g\in A_{n}$ where $n\geq p$ and hence
$g$ fixes $n-|\mu(g)|$ letters and $n-|\mu(g)|\geq p-|\mu(g)|$.
Taking any of these free letters as necessary, it follows that $\pw(g)\leq2$
by Theorem \ref{thm:(Bertram):-Regular even Bertram}. 

The remainder of this proof now addresses cases (1)-(5) of Lemma \ref{lem:first cd lemma}.
In particular we explain where to find the necessary free extra letters
(to apply weak Theorem \ref{thm:(Bertram):-Regular even Bertram}
to $g$) when we consider the 5 different possible sub-elements $h$
of $\prod_{i=1}^{k_{a}}a_{i}\prod_{j=1}^{k_{c}}c_{j}$. 

\textbf{Case (1):} Assume that $h=a$, a single odd length cycle such
that $p\leq|\mu(a)|\leq2p-1$. 

Firstly assume additionally that $a$ is not a $p$-cycle. Let $g=A_{1}B_{1}$
be the weak application of Theorem \ref{thm:(Bertram):-Regular even Bertram}
requiring $p-|\mu(g)|$ free letters. Write $|\mu(a)|=p+x$ and assume
additionally that $x\geq p-|\mu(g)|$. By Lemma \ref{lem:a cycles lemma}
case (5), we write a strong decomposition $a=A_{2}B_{2}$ where $|\mu(B_{2})\backslash\mu(A_{2})|=x$.
We then take these free letters in the decomposition of $g$ to write
$ga=A_{1}\cdot B_{1}A_{2}\cdot B_{2}\in I_{p}(A_{n})^{3}$, and this is
a strong decomposition as required.

Suppose instead that $x<p-|\mu(g)|$. First split off a $p$-cycle
$P$ from $a$ i.e. $ga=ga'P$ where $|\mu(a')\cap\mu(P)|=1$ and
$|\mu(a')|=x+1$. Now there exist two possibilities: if $|\mu(ga')|+c^{*}(ga')\leq2p$
then the statement of Lemma \ref{lem:first cd lemma} follows by an
application of Theorem \ref{thm:(Bertram):-Regular even Bertram}
(any free letters needed for a weak decomposition can be taken from
the $p-1$ available in $\mu(P)$). So suppose instead that 
\[
\frac{|\mu(ga')|+c^{*}(ga')}{2}=\frac{|\mu(g)|+c^{*}(g)+(x+2)}{2}>p.
\]
 But note that $|\mu(g)|$ is at most $p-x-1$ and so 
\[
c^{*}(g)>2p-(x+2)-|\mu(g)|\geq p-1.
\]
 This is clearly a contradiction as $|\mu(g)|\leq p-1$ implies $c^{*}(g)\leq\frac{p-1}{2}$.

When $a$ is a $p$-cycle it is clear that Lemma \ref{lem:first cd lemma}
still holds as we apply weak Theorem \ref{thm:(Bertram):-Regular even Bertram}
to $g$, taking any necessary extra letters from those in $\mu(a)$.

\textbf{Case} \textbf{(2)}: Assume that $h=a$ is a single odd length
cycle such that $2p\leq|\mu(a)|$.

By Lemma \ref{lem:a cycles lemma} we write $ag=y_{1}y_{2}y_{3}g$
where the $y_{i}\in I_{p}(A_{n})$ and $\mu(y_{i})\subseteq\mu(a)$. As
$y_{1}$ contains at least $p-2$ free letters (see Lemma \ref{lem:a cycles lemma})
we can choose from these when we form the weak decomposition $g=AB$.
It follows that $ag=y_{1}\cdot y_{2}A\cdot y_{3}B$ is a strong decomposition
in $I_{p}(A_{n})^{3}$.

\textbf{Case (3):} Assume that $h=b_{1}b_{2}$, a pair of even length
cycles such that $|\mu(b_{1})|\geq2p+2$ and $|\mu(b_{2})|>p$.

Firstly, by Lemma \ref{lem:b cycles lemma} Case (1) we write $b_{1}b_{2}g=y_{1}y_{2}y_{3}g$
such that $y_{i}\in I_{p}(A_{n})$, $\mu(y_{i})\subseteq\mu(b_{1}b_{2})$
and there exist at least $p-2$ free letters in either $y_{1}$ or
$y_{2}$. We then apply Theorem \ref{thm:(Bertram):-Regular even Bertram}
to write $g=AB$ using any of these necessary free letters. The result
is however not immediate if they have been taken from $\mu(y_{2})$
as this means $\mu(y_{2})\cap\mu(A)\neq\emptyset$ and it isn't
clear that $b_{1}b_{2}g=y_{1}(y_{2}A)(y_{3}B)$ is a product of $3$
$\op$-elements. Furthermore, $y_{2}$ and $A$ will not commute in
general, but we can work around this in the following way. 

Firstly note from the proof of Lemma \ref{lem:b cycles lemma} Case
(1) that the free letters (in $\mu(y_{2})$) we are using in the decomposition
$g=AB$ are all taken from a single cycle of $y_{2}$. Therefore $A$
commutes with all but at most 1 $p$-cycle of $y_{2}$, and so we can
reduce to considering the case where $y_{2}$ is indeed a single $p$-cycle.
Now we can assume without loss of generality that the $k=p-|\mu(g)|$
free letters we use appear successively in $y_{2}$. We therefore
denote these by $o_{1},\dots,o_{k}\in\mu(y_{2})$ and write $y_{2}=(o_{1},\dots,o_{k},l_{1},\dots,l_{p-k})$.
Also, by Remark \ref{rem:successively} it is clear that $A$ and $B$ are
constructed such that any letters in $\mu(A)\backslash\mu(g)$ appear
successively in the cycle $A$. In summary we can write $A=(o_{k},\dots,o_{1},m_{1},\dots,m_{p-k})$,
where $\mu(g)=\{m_{i}\}_{i=1}^{p-k}$. It is then a straightforward
check that 
\begin{eqnarray*}
y_{2}\cdot A & = & (o_{1},\dots,o_{k},l_{1},\dots,l_{p-k})(o_{k},\dots,o_{1},m_{1},\dots,m_{p-k})\\
 & = & (o_{k},l_{1},\dots,l_{p-k},m_{1},\dots,m_{p-k})\\
 & = & (m_{1},\dots,m_{p-k},o_{k},o_{1},o_{2},\dots,o_{k-1})(m_{1},o_{k-1},\dots\,o_{1},l_{1},\dots,l_{p-k})\\
 & =: & A'\cdot y'_{2}.
\end{eqnarray*}
Note that $\mu(A)=\mu(A')$ and $|\mu(y_{2})\cap\mu(y'_{2})|=p-1$.
Furthermore, as our initial sub-element $g$ is disjoint from $b_{1}b_{2}$,
it follows that $m_{1}\notin\mu(y_{2})$ and that $y'_{2}$ is still
a $p$-cycle. We conclude that $b_{1}b_{2}g=y_{1}y_{2}y_{3}g=y_{1}y_{2}Ay_{3}B=y_{1}A'\cdot y'_{2}\cdot y_{3}B$
is a strong decomposition in $I_{p}(A_{n})^{3}$.

\textbf{Case (4): }Assume that $h=bd$ a pair of even length cycles
such that $|\mu(d)|<p<|\mu(b)|$. 

If $2p\leq|\mu(b)|$ then by Lemma \ref{lem:b cycles lemma} there
exists a strong decomposition of $bd\in I_{p}(A_{n})^{3}$. In particular
$bd=y_{1}\cdot y_{2}A_{1}\cdot y_{3}B_{1}$ where either $y_{1}$
or $y_{2}$ contains at least $p-2$ free letters. Let $g=A_{2}B_{2}$
be the weak decomposition into $p$-cycles by Theorem \ref{thm:(Bertram):-Regular even Bertram},
using at most $p-3$ of these available letters. If the letters are
taken from $\mu(y_{1})$ then $bdg = y_{1}\cdot y_{2}A_{1}A_{2} \cdot y_{3}B_{1}B_{2}$ is a strong decomposition in $I_{p}(G)^{3}$ as required. Otherwise, $y_{2}$ and $A_{2}$  don't necessarily commute but the result follows in a similar manner to the previous case.


Assume instead that $|\mu(b)|\leq2p-2.$ First split off a $p$-cycle
by writing $b=b'P$ where $|\mu(b')\cap\mu(P)|=1$. It follows that
$2\leq|\mu(d)|,\,|\mu(b')|\leq p-1$ and we now consider the product
$gdb'P$. As $|\mu(db')|\leq2p-2$ we can apply Theorem \ref{thm:(Bertram):-Regular even Bertram}.
Suppose additionally that $|\mu(db')|\geq p$ and so we can write
$db'=A_{1}B_{1}$, a strong decomposition of $p$-cycles. We then apply weak
Theorem \ref{thm:(Bertram):-Regular even Bertram} to give $g=A_{2}B_{2}$
where we take any necessary extra letters from $\mu(P)$. It follows
that $gdb=gdb'P=A_{2}B_{2}A_{1}B_{1}P=A_{2}A_{1}\cdot B_{2}B_{1}\cdot P\in I_{p}(A_{n})^{3}$, and
that this is a strong decomposition. 

It remains to consider the case where $|\mu(db')|<p$, and this divides
into 2 sub-cases. Firstly, if 
\begin{equation}
\frac{|\mu(gdb')|+c^{*}(gdb')}{2}\leq p\label{eq:gdb'}
\end{equation}
 then we apply Theorem \ref{thm:(Bertram):-Regular even Bertram}
to write $gdb'$ as a product of 2 $p$-cycles. We have $p-1$ available
free letters in $\mu(P)$ if needed and then the result follows similarly
to above. If instead (\ref{eq:gdb'}) does not hold then
\begin{equation}
|\mu(gdb')|+c^{*}(gdb')>2p\label{eq:decomp1}
\end{equation}
 and we consider separate weak decompositions of $g$ and $db'$.
Both of these are possible but we require a total of $2p-(|\mu(g)|+|\mu(d)|+|\mu(b')|)$
extra letters. If these exist in $\mu(P)$ then the conclusion of
the Lemma follows as above. If not, then $2p-(|\mu(g)|+|\mu(d)|+|\mu(b')|)\geq p.$
 From (\ref{eq:decomp1}) it then follows that 
\[
c^{*}(g)+2>2p-(|\mu(g)|+|\mu(d)|+|\mu(b')|)\geq p.
\]
 But this is a contradiction as $c^{*}(g)\leq\frac{p-1}{2}$ yields
that $\frac{p+3}{2}>p$ which is false. 

\textbf{Case (5):} Assume that $h=b_{1}b_{2}$, a pair of even length
cycles such that $p<|\mu(b_{i})|\leq2p$.

Recall from Lemma \ref{lem:b cycles lemma} that $\pw(b_{1}b_{2})\leq3$.
In particular there exists a strong decomposition $b_{1}b_{2}=A_{1}B_{1}c$,
where $A_{1}$ and $B_{1}$ are $p$-cycles appearing from an application
of Theorem \ref{thm:(Bertram):-Regular even Bertram} and $c$ is
a product of at most 3 $p$-cycles. We find in all cases that we have
at least $p-1$ free letters in $\mu(c)$ and thus we can use these
in the weak decomposition $g=A_{2}B_{2}$.
The result then follows in a similar manner to the above, completing the final case of Lemma \ref{lem:first cd lemma}.
\end{proof}

\subsection{Cycles in $\{c_{i}\}$ and pairs of cycles in $\{d_{i}\}$; final
considerations\label{sec:Can't apply Bertram}}

In this final section we consider sub-elements of $\sigma$ of the
form $g=g_{1}\dots g_{k}=\prod_{j=1}^{k_{c}}c_{j}\prod_{t=1}^{k_{d}}d_{t}$
such that $|\mu(g)|+c^{*}(g)>2p$.  This condition prevents us from
applying Theorem \ref{thm:(Bertram):-Regular even Bertram}, but also
note that we may not be able to apply strong Theorem \ref{thm:(Bertram):-Regular even Bertram}
to a sub-element of $g$. For example, consider $p=5$ and a $(2^{4})$
element in $A_{8}$. 

However, we claim that there exists a sub-collection of cycles $\{g_{i_{j}}\}\subset\{g_{i}\}_{i=1}^{k}$
such that we can apply either strong Theorem \ref{thm:(Bertram):-Regular even Bertram}
or strong Theorem \ref{thm:(Bertram):odd bertram} to the sub-element
$\Pi g_{i_{j}}$. For example in the case above we can apply strong Theorem
\ref{thm:(Bertram):odd bertram} to the sub-element consisting of
a $(2^{3})$ element. We formalise and prove this statement below.
\begin{lem}
Let $g=g_{1}\dots g_{k}=\prod_{j=1}^{k_{c}}c_{j}\prod_{t=1}^{k_{d}}d_{t}\in S_{n}$
where the $c_{j}$ and $d_{t}$ are disjoint cyc\label{lem:We can apply odd or even Bertram(strong) to subelements}les
of odd and even lengths respectively, and all lengths are less than
$p$. Further assume that
\begin{equation}
|\mu(g)|+c^{*}(g)>2p\ \ \text{if}\ g\in A_{n},\label{eq:can't apply even bertram}
\end{equation}
\begin{equation}
|\mu(g)|+c^{*}(g)>2p-1\ \ \text{if}\ g\in S_{n}\backslash A_{n}.\label{eq:can't apply odd bertrram}
\end{equation}
Then there exists a sub-element $g'$ of $g$, such that $|\mu(g')|\geq p$
and 
\begin{equation}
|\mu(g')|+c^{*}(g')\leq2p\ \ \text{if}\ g'\in A_{n},\label{eq:can apply even bertram}
\end{equation}
\begin{equation}
|\mu(g')|+c^{*}(g')\leq2p-1\ \ \text{if}\ g'\in S_{n}\backslash A_{n}.\label{eq:can apply odd bertram}
\end{equation}
 In particular we can apply either strong Theorem \ref{thm:(Bertram):-Regular even Bertram}
or strong Theorem \ref{thm:(Bertram):odd bertram} to $g'$.
\end{lem}
\begin{proof}
We proceed by induction on $k$. Clearly our base case is $k=3$:
if $k=1,2$ then the assumption that $|\mu(g_{i})|<p$ contradicts
(\ref{eq:can't apply even bertram}) or (\ref{eq:can't apply odd bertrram}). 

Firstly assume that $g_{1}g_{2}g_{3}\in A_{n}$ 
and consider the sub-element $g_{1}g_{2}\in S_{n}$. If $g_{1}g_{2}$ satisfies
either condition (\ref{eq:can apply even bertram}) or (\ref{eq:can apply odd bertram})
then the conclusion of the lemma follows. We therefore assume that
this is not the case and so we must have insufficient letters i.e.
$|\mu(g_{1}g_{2})|<p$. Recall that by assumption, $|\mu(g_{3})|\leq p-1$ and hence
$|\mu(g_{1}g_{2}g_{3})|\leq2p-2$. It then follows from the  initial assumption
(\ref{eq:can't apply even bertram}) that $|\mu(g_{1}g_{2}g_{3})|=2p-2$
exactly, and so $|\mu(g_{1}g_{2})|=p-1=|\mu(g_{3})|$. Noting that $|\mu(g_{1})|\leq p-3$, it is clear that the conclusion of the lemma holds for the sub-element $g_{1}g_{3}\in S_{n}$. The case $g_{1}g_{2}g_{3}\in S_{n}\backslash A_{n}$
follows in an almost identical manner.

For the inductive step assume that the result holds for $k\geq3$
and first consider $g=g_{1}\dots g_{k+1}\in A_{n}$. We claim that
there exists $1\leq i\leq k+1$ such that $g'=g\backslash g_{i}$
(see Definition \ref{def:Initial definitions of An}) has $|\mu(g')|\geq p$:
to see this first note that if all such sub-elements $g'$ have less
than $p$ letters, then in particular $|\mu(g_{1}\dots g_{k})|<p$.
Combining this with the initial assumption (\ref{eq:can't apply even bertram})
that $|\mu(g_{1}\dots g_{k+1})|+(k+1)>2p$ yields 
\[
|\mu(g_{k+1})|>2p-(k+1)-|\mu(g_{1}\dots g_{k})|>p-(k+1).
\]
 Similarly $|\mu(g_{2}\dots g_{k+1})|<p$, and noting that $|\mu(g_{i})|\geq2$
for all $g_{i}$ gives
\[
|\mu(g_{k+1})|<p-2(k-1).
\]
 Hence 
\[
p-2k+2>p-k-1,
\]
 which is a contradiction, proving the claim.

Fix such a $g'$ that moves at least $p$ letters. Now either $g'$
satisfies the conclusion of the lemma or by the inductive hypothesis
there exists a sub-element of $g'$ which does. Naturally the result
follows for $g$ in either case. If $g\in S_{n}\backslash A_{n}$
we use a similar method and the proof of the lemma is complete.
\end{proof}
Using Lemma \ref{lem:We can apply odd or even Bertram(strong) to subelements}
we can now prove the main result of this section.
\begin{lem}
\label{lem:cd cycles second lemma}Let $g=\prod_{s=1}^{k_{c}}c_{s}\prod_{t=1}^{k_{d}}d_{t}$
be the sub-element of $\sigma$ as in (\ref{eq:disjoint decomp of our initial element})
and assume that $(|\mu(g)|+c^{*}(g))/2>p$. Then $\pw(g)\leq3$ and
$g$ has a strong decomposition in $I_{p}(A_{n})^{3}$. 
\end{lem}
\begin{proof}
 The first step in the proof of Lemma \ref{lem:cd cycles second lemma}
is to repeatedly apply Lemma \ref{lem:We can apply odd or even Bertram(strong) to subelements}
to our element $g=\prod_{j=1}^{k_{c}}c_{j}\prod_{t=1}^{k_{d}}d_{t}$.
This reduces it to the form 
\begin{equation}
g=A_{1}\dots A_{u}C_{1}\dots C_{v}\cdot B_{1}\dots B_{u}d_{1}\dots d_{v}\cdot g_{1}\dots g_{t},\label{eq:reference decomp after iterations}
\end{equation}
 where $A_{i}B_{i}$ are pairs of $p$-cycles and $C_{j}d_{j}$ are
products of a $p$-cycle and $(p-1)$-cycle. These come from applying
strong Theorem \ref{thm:(Bertram):-Regular even Bertram} and strong
Theorem \ref{thm:(Bertram):odd bertram} respectively. Let $r=g_{1}\dots g_{t}\in S_{n}$
denote the ``remainder'' from this process, chosen such that $r$
satisfies either condition (\ref{eq:can apply even bertram}) or (\ref{eq:can apply odd bertram}).
In particular we can apply either Theorem \ref{thm:(Bertram):-Regular even Bertram}
or Theorem \ref{thm:(Bertram):odd bertram} depending on the parity
of $r$.

Firstly suppose that $r=1$. As $g\in A_{n}$ it follows that $v$
is even and we can apply strong Theorem \ref{thm:(Bertram):-Regular even Bertram}
to pairs $d_{2i+1}d_{2i+2}$ to give $d_{2i+1}d_{2i+2}=P_{i}P'_{i}\in I_{p}(A_{n})^{2}.$
Rearranging cycles then yields
\begin{align*}
g & =A_{1}\dots A_{u}C_{1}\dots C_{v}\cdot B_{1}\dots B_{u}d_{1}\dots d_{v}\\
 & =A_{1}\dots A_{u}C_{1}\dots C_{v}\cdot B_{1}\dots B_{u}P_{1}\dots P_{v/2}\cdot P'_{1}\dots P'_{v/2}.
\end{align*}
 Hence $g\in I_{p}(A_{n})^{3}$ with a strong decomposition as required.
If instead $r$ is non-trivial and additionally $|\mu(r)|\geq p$
then we first apply strong Theorem \ref{thm:(Bertram):-Regular even Bertram}
or \ref{thm:(Bertram):odd bertram} to $r$ before proceeding as above
(the $r=1$ case).

To complete the proof of Lemma \ref{lem:cd cycles second lemma} it
remains to consider $r\neq1$ such that $|\mu(r)|<p$. Now we can
apply weak Theorem \ref{thm:(Bertram):-Regular even Bertram} or \ref{thm:(Bertram):odd bertram}
to $r=g_{1}\dots g_{t}$ but we will need up to $p-2$ free letters
to complete the decomposition. As in the proof of Lemma \ref{lem:first cd lemma},
we now have a number of sub-cases to consider and for each of these
we show where to find the necessary letters. 

\textbf{Case (1): }Assume that the decomposition (\ref{eq:reference decomp after iterations})
contains at least one pair $A_{1}B_{1}$: 

Firstly assume that $v$ is even. We apply strong Theorem \ref{thm:(Bertram):-Regular even Bertram}
to the pairs of $p-1$ cycles, writing $d_{2i+1}d_{2i+2}=P_{i}P'_{i}$
and rearranging cycles in $g$ to give 
\[
g=A_{1}B_{1}r\cdot A_{2}\dots A_{u}C_{1}\dots C_{v}\cdot B_{2}\dots B_{u}P_{1}\dots P_{v/2}\cdot P'_{1}\dots P'_{v/2}.
\]
If instead we assume that $v$ is odd then $g_{1}\dots g_{t}\in S_{n}\backslash A_{n}$.
Without loss of generality $|\mu(g_{1})|$ is even and we relabel
$g_{1}=d_{v+1}$ before proceeding as above, writing pairs $d_{2i+1}d_{2i+2}$
as products of $p$-cycles to achieve an analogous expression for
$g$.

Now the letters $\mu(A_{1}B_{1}r)$ are disjoint from the other cycles
so it suffices to show that $A_{1}B_{1}r\in I_{p}(A_{n})^{3}$ (using
only letters in $\mu(A_{1}B_{1}r)$). For ease of notation we drop
the subscript from $A_{1}B_{1}$. 

Firstly recall that we can still apply weak Theorem \ref{thm:(Bertram):-Regular even Bertram}
to write $r=PP'$, a product of $p$-cycles. If there exists a sufficient
number of free letters in the $p$-cycle $A$ (so that $B$ and $P$
are disjoint), that is $|\mu(AB)|-p\geq p-|\mu(r)|,$ then $ABr=A\cdot BP\cdot P'\in I_{p}(A_{n})^{3}$
is a strong decomposition. We can therefore assume otherwise. It may
also hold that $Br$ satisfies the conditions of strong Theorem \ref{thm:(Bertram):-Regular even Bertram}
i.e. 
\[
|\mu(B)|+|\mu(r)|+(c^{*}(r)+1)\leq2p.
\]
Here we would clearly also have a strong decomposition of $ABr\in I_{p}(A_{n})^{3}$.
We shall therefore again assume otherwise, and in summary 
\[
|\mu(AB)|+|\mu(r)|<2p,
\]
\[
|\mu(r)|+c^{*}(r)\geq p.
\]
 Recall also that $3\leq|\mu(r)|<p$ and therefore it follows that
$p\leq|\mu(AB)|<2p-3$. 

Now if $|\mu(r)|<\frac{2p}{3}$ then
\[
|\mu(r)|+c^{*}(r)<\frac{3}{2}\cdot\frac{2p}{3}=p,
\]
 contradicting the above. In summary, we have the following information
\begin{equation}
p\leq|\mu(AB)|,\ \frac{2p}{3}\leq|\mu(r)|<p,\label{eq:Assumption 1}
\end{equation}
\begin{equation}
|\mu(AB)|+|\mu(r)|<2p,\label{eq:Assumption 2}
\end{equation}
\begin{equation}
p\leq|\mu(r)|+c^{*}(r)\leq\frac{3}{2}(p-1).\label{eq:Assumption 3}
\end{equation}
 It is now based on these assumptions that we show that $\pw(ABr)\leq3$
and that the decomposition in $I_{p}(A_{n})^{3}$ is strong. Firstly note that any product of disjoint cycles $(a_{i,1}\dots a_{i,l_{i}})$,
$1\leq i\leq k$ can be rewritten as 
\begin{equation}
\prod_{i=1}^{k}(a_{i,1}\dots a_{i,l_{i}})=(a_{1,1}\dots a_{k,l_{k}})(a_{k,1}\,a_{k-1,1}\,\dots a_{1,1}).\label{eq:long element and fixing cycle}
\end{equation}
This is a product of a single cycle containing all the letters of
$\mu\left(\Pi_{i=1}^{k}(a_{i,1}\dots a_{i,l_{i}})\right)$ in original
order, plus a "fixing" end cycle. Applying this to our element gives
$ABr=ly$ such that $\mu(l)=\mu(ABr)$ and $\mu(y)=c^{*}(AB)+c^{*}(r)$.
Now by assumption (\ref{eq:Assumption 1}), $|\mu(l)|\geq p+3$ and
hence we can break off a $p$-cycle $P$, writing $ABr=ly=Pxy$ where
$|\mu(P)\cap\mu(x)|=\{a_{1,1}\}.$ We claim that $xy$ satisfies the
conditions of Theorem \ref{thm:(Bertram):-Regular even Bertram} i.e.
\begin{equation}
|\mu(xy)|+c^{*}(xy)\leq2p.\label{eq:A claim}
\end{equation}
 To see this note that 
\begin{eqnarray*}
|\mu(xy)| & =  & |\mu(x)|+|\mu(y)|-|\mu(x)\cap\mu(y)|\\
 & = & \left(|\mu(AB)|+|\mu(r)|-(p-1)\right)+\left(c^{*}(AB)+c^{*}(r)\right)-|\mu(x)\cap\mu(y)|.
\end{eqnarray*}
 As $a_{1,1}\in\mu(x)\cap\mu(y)$, it follows that $c^{*}(xy)\leq|\mu(x)\cap\mu(y)|$
by Lemma \ref{lem:lemma}. Hence 
\begin{equation}
|\mu(xy)|+c^{*}(xy)\leq(\mu|(AB)|+|\mu(r)|-(p-1))+c^{*}(AB)+c^{*}(r).\label{eq:almost-1}
\end{equation}
 Denote $m=|\mu(A)\cap\mu(B)|$. It's clear that $m\geq1$ as by construction
$A$ and $B$ cannot be disjoint $p$-cycles.  It follows that $|\mu(AB)|= 2p-m$
and $c^{*}(AB)\leq$ min$\{m,\,\frac{2p-m}{2}$\} by Lemma \ref{lem:lemma}.
We also note by assumption (\ref{eq:Assumption 2}) that $|\mu(r)|\leq m-1$
and hence $c^{*}(r)\leq\frac{m-1}{2}.$ Combining this all with (\ref{eq:almost-1})
yields 
\begin{eqnarray*}
|\mu(xy)|+c^{*}(xy) & \leq & (2p-m)+m-1-(p-1)+\text{min}\{m,\,\frac{2p-m}{2}\}+\frac{m-1}{2}\\
 & = & p+\text{min}\{m,\,\frac{2p-m}{2}\}+\frac{m-1}{2}.
\end{eqnarray*}
 Now $m\geq\frac{2p-m}{2}$ if and only if $m\geq\frac{2p}{3}$. If
this holds then
\[
|\mu(xy)|+c^{*}(xy)\leq p+\frac{2p-m}{2}+\frac{m-1}{2}\leq p+\frac{2p-1}{2}<2p.
\]
 Suppose instead that $m\leq\frac{2p}{3}$. Then again
\[
|\mu(xy)|+c^{*}(xy)\leq p+\frac{3m-1}{2}\leq p+\frac{2p-1}{2}<2p,
\]
 and this proves the claim (\ref{eq:A claim}). As a result of claim
(\ref{eq:A claim}) we now apply Theorem \ref{thm:(Bertram):-Regular even Bertram}
and write $xy=A'B'$, a product of two $p$-cycles. This then yields
$ABr=PA'B'\in I_{p}(A_{n})^{3}$. Naturally if $|\mu(xy)|\geq p$ then
$xy=A'B'$ is a strong decomposition. If instead $xy=A'B'$ is a weak
decomposition, we have no issues as we know there exists a sufficient
number of free letters ($|\mu(ABr)|>|\mu(AB)|\geq p$ by assumption
(\ref{eq:Assumption 1})).

\textbf{Case (2): }Assume that the decomposition (\ref{eq:reference decomp after iterations})
contains no pairs $A_{i}B_{i}$: 

In this case we have only applied Theorem \ref{thm:(Bertram):odd bertram}
to sub-collections of our initial element, yielding
\[
g=C_{1}\dots C_{v}\cdot d_{1}\dots d_{v}\cdot r.
\]
Recall that $|\mu(r)|<p$, $C_{i}$ are mutually disjoint $p$-cycles
and $d_{i}$ are mutually disjoint $(p-1)$-cycles. Also note that
sgn$(r)=(-1)^{v}$, where $v\geq1$ (the case $v=0$ is addressed
by Lemma \ref{lem:first cd lemma}). 

In a similar fashion to the previous case we first reduce the number
of pairs $d_{i}$: if $v$ is even then we group the pairs $d_{3}d_{4}$,
$d_{5}d_{6},\dots$ and apply Theorem \ref{thm:(Bertram):-Regular even Bertram}
to write these as strong decompositions of $p$-cycles $d_{2i+1}d_{2i+2}=P_{i}P'_{i}$.
Rearranging gives
\[
g=C_{1}C_{2}d_{1}d_{2}r\cdot C_{3}\dots C_{v}\cdot P_{2}\dots P_{v/2}\cdot P'_{2}\dots P'_{v/2}.
\]

As $C_{1}C_{2}d_{1}d_{2}r$ is disjoint from the rest of $g$, it suffices
for the lemma to show that $C_{1}C_{2}d_{1}d_{2}r\in I_{p}(A_{n})^{3}$
using only the letters in $\mu(C_{1}C_{2}d_{1}d_{2}r)$. If instead
$v$ is odd, then we similarly pair $(p-1)$-cycles $d_{2}d_{3}\dots$
together and reduce to the case $g=C_{1}d_{1}r$. Note that here
$r\in S_{n}\backslash A_{n}$.

So in summary we now consider 
\begin{equation}
C_{1}C_{2}d_{1}d_{2}r,\ \ r\in A_{n}\label{eq:double C case}
\end{equation}
 and 
\begin{equation}
C_{1}d_{1}r,\ \ r\in S_{n}\backslash A_{n}.\label{eq:single C case}
\end{equation}
Evidently, if there exists a pair of even length cycles in $r\in A_{n}$ then
we can write $r=hh'$ where $h,\,h'\in S_{n}\backslash A_{n}$ and
then split case (\ref{eq:double C case}) into two separate examples
of (\ref{eq:single C case}). The latter will therefore be the main
focus of the following work. We first however treat the case where $r$ consists purely of odd-length cycles and this
splitting procedure is not possible.

So let $z=C_{1}C_{2}d_{1}d_{2}r$ and also assume initially that $r=g_{1}$,
a single odd-length cycle. First apply strong Theorem \ref{thm:(Bertram):-Regular even Bertram}
to write $d_{1}d_{2}=PP'$ a product of two $p$-cycles. As $\mu(d_{1}d_{2})=2p-2$
it follows that $\mu(P)\cap\mu(P')=\{a,\,b\}$ and we choose these letters to
occur successively in both $P$ and $P'$ (for example when $p=7$: 
\begin{eqnarray}
d_{1}d_{2} & = & (1\,2\,3\,4\,5\,6)(7\,8\,9\,11\,10\,12)\label{eq:d1d2 as p cycles}\\
 & = & (1\,2\,3\,4\,5\,6\,7\,8\,9\,10\,11\,12)(7\,1)\nonumber \\
 & = & (1\,2\,3\,4\,5\,6\,7)(7\,1\,8\,9\,10\,11\,12)).\nonumber 
\end{eqnarray}
This yields $z=C_{1}C_{2}\cdot P\cdot rP'$ and we then apply Theorem
\ref{thm:(Bertram):-Regular even Bertram} again to write $rP'=AB$,
a strong decomposition of $p$-cycles. To show that $z\in I_{p}(A_{n})^{3}$
we want to group the $p$-cycles $P$ and $A$ to give $z=C_{1}C_{2}\cdot PA\cdot B$.
This is clearly possible if $\{a,\,b\}\cap A=\emptyset$, and $A$ can be naturally chosen this way as the letters $a,\,b$
appear successively in $P'$ and $|\mu(r)|\geq3$.  For example if $p=7$, $|\mu(r)|=5$ and $\{a,\,b\}=\{11,12\}$:
\begin{eqnarray}
rP' & := & (1\, 2\, 3\, 4\, 5)(6\, 7\, 8\, 9\, 10\, 11\, 12) \nonumber \\
 & = & (1\, 2\, 3\, 4\, 5\, 6\, 7)(6\, 1\, 8\, 9\, 10\, 11\, 12)\nonumber \\
 & =: & AB.\nonumber
\end{eqnarray}


Now consider $z=C_{1}C_{2}d_{1}d_{2}r=C_{1}C_{2}PP'g_{1}\dots g_{t}$,
where $g_{i}$ are cycles of odd length less than $p$ and $t\geq2$.
We will see that $\pw(z)\leq3$ (and thus the conclusion of Lemma
\ref{lem:cd cycles second lemma} for Case (\ref{eq:double C case}))
can be deduced from the following claim.

\textit{Claim:}\textbf{ }There exist two sub-elements $h$ and $h'$
of $r$ such that $r=hh'$, $|\mu(hP)|+c^{*}(hP)\leq2p$ and $|\mu(hP)|+c^{*}(hP)\leq2p$. 

\textit{Proof.} Firstly, assume that there exists $g_{i}$, a sub-element
of $r$ such that $|\mu(g_{i})|\geq\frac{|\mu(r)|}{2}$. A simple
check shows that the elements $h=g_{i}$ and $h'=r\backslash g_{i}$
will then suffice for the claim (recall that $|\mu(r)|<p$). We therefore
assume from now on that the contrary holds, that is that $|\mu(g_{i})|<\frac{|\mu(r)|}{2}$
for $1\leq i\leq t$ (and consequently $t\geq 3$). Assume also for a contradiction that no such
$h$ and $h'$ exist. We can therefore instead take a pair $r=hh'$ such that without
loss of generality, the element $hP$ does not satisfy the claim and it follows
that 
\begin{equation}
|\mu(h)|+c^{*}(h)>p-1.\label{eq:h fails-1}
\end{equation}
 By this choice, we can then check that $h'P'$ does satisfy
the claim: if $|\mu(h')|+c^{*}(h')>p-1$ as well, then $|\mu(r)|+c^{*}(r)>2p-2$
which is false. So in summary $hP$ fails the claim but $h'P'$ does not. Finally,
we may also assume that $h$ is chosen to be ``just failing'', that
is if we remove any cycle $g_{i}$ from $h$ and attach it to the
sub-element $h'$ then the claim will apply 
to $h\backslash g_{i}$. Consequently, $h'g_{i}$ fails for all $g_{i}$
in $h$ i.e. 
\begin{equation}
|\mu(h'g_{i})|+c^{*}(h'g_{i})=|\mu(h')|+c^{*}(h')+|\mu(g_{i})|+1>p-1.\label{eq:h'g_i fails-1}
\end{equation}
 Combining (\ref{eq:h fails-1}) and (\ref{eq:h'g_i fails-1}) gives
\[
|\mu(r)|+c^{*}(r)+|\mu(g_{i})|>2p-3,
\]
for all $g_{i}$ in $h$. But recall that $|\mu(r)|<p$ (see the beginning
of Case (2)), hence $|\mu(r)|+c^{*}(r)\leq\frac{3(p-1)}{2}$ and 
\[
|\mu(g_{i})|>2p-3-\frac{3(p-1)}{2}=\frac{p-3}{2}.
\]
 This is a contradiction as $|\mu(g_{i})|<\frac{|\mu(r)|}{2}\leq\frac{p-1}{2}$
for all $g_{i}$ and this completes the proof of the claim.

To now complete the proof of Lemma \ref{lem:cd cycles second lemma}
Case (\ref{eq:double C case}) we apply the claim to find appropriate sub-elements
$h$ and $h'$ of $r$, and write $z=C_{1}C_{2}\cdot hP\cdot h'P'$.
Note that by construction (see (\ref{eq:d1d2 as p cycles})) $\mu(P)$
and $\mu(P')$ intersect in two letters, say $\{a,b\}$, and that
these occur successively in both $P$ and $P'$. It then follows from
the proof of Theorem \ref{thm:(Bertram):-Regular even Bertram} \cite[Thm. 1]{Bertram 1972}, that $hP=AB$ and $h'P'=A'B'$ where the $p$-cycles can be
chosen such that $\{a,b\}\cap(\mu(B)\cup\mu(A'))=\emptyset$. Recall
that $\mu(h)\cap\mu(h')=\emptyset$, hence it follows that $B$ and
$A'$ commute and $z=C_{1}C_{2}\cdot AA'\cdot BB'\in I_{p}(A_{n})^{3}$
is a strong decomposition. 

In summary, we have shown that $\pw(z)\leq3$ where $z=C_{1}C_{2}d_{1}d_{2}r$ and $r$ consists of odd length cycles. Furthermore, we can choose
a strong decomposition of $z$ in $I_{p}(A_{n})^{3}$.

To finish the proof of Lemma \ref{lem:cd cycles second lemma} Case
(2) we now consider sub-case (\ref{eq:single C case}) where $z=C_{1}d_{1}g_{1}\dots g_{t}$,
$C_{1}$ is a $p$-cycle, $d_{1}$ is a $(p-1)$-cycle and $r=g_{1}\dots g_{t}\in S_{n}\backslash A_{n}$
such that $|\mu(r)|<p$. That $\pw(z)\leq3$ follows in a very similar
manner to the above:

Firstly suppose that we can apply strong Theorem \ref{thm:(Bertram):-Regular even Bertram}
to $d_{1}r.$ Writing $d_{1}r=AB\in I_{p}(A_{n})^{2}$ clearly gives $C_{1}d_{1}r=C_{1}\cdot A\cdot B\in I_{p}(A_{n})^{3}$, and this is a strong decomposition as required. Now this is possible
only if 
\[
|\mu(d_{1}r)|+c^{*}(d_{1}r)\leq2p,
\]
 which holds if and only if 
\[
|\mu(r)|+c^{*}(r)\leq p.
\]
 Henceforth we may therefore assume that 
\begin{equation}
|\mu(r)|+c^{*}(r)\geq p+1.\label{eq:mu+c}
\end{equation}
 Another possibility is that there exists a sufficient number of free
letters in $\mu(d_{1})$ to apply weak Theorem \ref{thm:(Bertram):odd bertram}
to $r$, and then finish with an application of strong Theorem \ref{thm:(Bertram):-Regular even Bertram}:
 first rearrange our element to consider $C_{1}rd_{1}$. Assuming
initially that $p-|\mu(r)|\leq|\mu(C_{1}d_{1})|-p$, we take free
letters from $d_{1}$ to apply weak Theorem \ref{thm:(Bertram):odd bertram},
writing $r=C'd'$, a product of a $p$-cycle and a $(p-1)$-cycle.
Note that this choice of free letters guarantees that $\mu(C_{1})\cap\mu(C')=\emptyset$
and hence $C_{1}C'\in I_{p}(A_{n})$. Next we apply strong Theorem \ref{thm:(Bertram):-Regular even Bertram}
to $d'd_{1}$. That this is possible is clear if $d_{1}$ and $d'$
are disjoint, but if $|\mu(d_{1})\cap\mu(d')|=m\geq1$ then we note
by lemma \ref{lem:lemma} that
\begin{eqnarray*}
|\mu(d'd_{1})|+c^{*}(d'd_{1}) & \leq & |\mu(d')|+|\mu(d_{1})|-m+c^{*}(d'd_{1})\\
 & = & 2p-2-m+c^{*}(d'd_{1})\\
 & \leq & 2p-2-m+m.
\end{eqnarray*}
 It is indeed a strong decomposition as
 \begin{eqnarray*}
|\mu(d'd_{1})| & \geq & |\mu(d')|+|\mu(d_{1})|-2m\\
 & = & 2p-2-2(p-|\mu (r)|-1)\\
 & = & 2|\mu(r)|\\
 & \geq & 4(p+1)/3,
\end{eqnarray*} 
by (\ref{eq:mu+c}).
So in conclusion, we decompose $d'd_{1}$ into $p$-cycles $PP'$
yielding a final strong decomposition 
\[
C_{1}rd_{1}=C_{1}C'd'd_{1}=C_{1}C'\cdot P\cdot P'\in I_{p}(A_{n})^{3}.
\]
As the above assumed that $p-|\mu(r)|\leq|\mu(C_{1}d_{1})|-p$, we
now assume the contrary and so we know the following about our element
$z=C_{1}d_{1}r$: 
\begin{equation}
p\leq|\mu(C_{1}d_{1})|\leq2p-3,\ \ 2\leq|\mu(r)|\leq p-1\label{eq:Assumption 1'}
\end{equation}
 
\begin{equation}
p+1\leq|\mu(r)|+c^{*}(r)\leq\frac{3(p-1)}{2}\label{eq:assumption 3'}
\end{equation}
 
\begin{equation}
|\mu(C_{1}d_{1})|+|\mu(r)|<2p\label{eq:assumption 4'}
\end{equation}
 Now as before (see (\ref{eq:long element and fixing cycle}) in Case
(1)) we write $C_{1}d_{1}r$ as a product of a single cycle containing
all letters in $\mu(C_{1}d_{1}r)$ and a ``fixing'' end cycle. Denote
the long element by $l$ and the end piece by $y$. This gives $C_{1}d_{1}r=ly$
such that $|\mu(l)|=|\mu(C_{1}d_{1}r)|$ and $|\mu(y)|=c^{*}(C_{1}d_{1})+c^{*}(r)$.
Now by assumption (\ref{eq:Assumption 1'}), $|\mu(C_{1}d_{1}r)|\geq p+2$
hence we can break off a $p$-cycle $P$, writing $ly=Pxy$ where
$|\mu(P)\cap\mu(x)|=\{a_{1,1}\}.$ We claim that we can apply Theorem
\ref{thm:(Bertram):-Regular even Bertram} to $xy$ and for this we
require 
\[
|\mu(xy)|+c^{*}(xy)\leq2p.
\]

This however follows in an identical fashion to Case (1) of this Lemma and so we omit the details. In summary,  we apply Theorem \ref{thm:(Bertram):-Regular even Bertram}
to write $xy=AB$, a product of two $p$-cycles. If $|\mu(xy)|\geq p$
then the result follows as this is a strong decomposition. Our only
remaining case is therefore when $|\mu(xy)|<p$, but this is no issue
as any extra letters we need can be taken from the cycle $P$.

This completes the proof of Lemma\textbf{ }\ref{lem:cd cycles second lemma}.
\end{proof}

\subsection{Proof of Theorem \ref{thm:my theorem} }
\begin{proof}
(Of Theorem \ref{thm:my theorem}). Firstly, if $p=2$ then the result is given by \cite[2.3]{AJM 2017}. Hence we may assume that $p$ is an odd prime. Let $\sigma\in A_{n}$ and consider
the disjoint cycle decomposition $\sigma=\prod_{i=1}^{k_{a}}a_{i}\prod_{j=1}^{k_{c}}c_{j}\prod_{s=1}^{k_{b}}b_{s}\prod_{t=1}^{k_{d}}d_{t}$
as in (\ref{eq:disjoint decomp of our initial element}). Let $g=\prod_{j=1}^{k_{c}}c_{j}\prod_{t=1}^{k_{d}}d_{t}\in A_{n}$
and assume initially that $|\mu(g)|+c^{*}(g)>2p$. By Lemma \ref{lem:cd cycles second lemma}
we have that $g\in I_{p}(A_{n})^{3}$ and we can choose a strong decomposition
for $g$. Similarly, by Lemmas \ref{lem:a cycles lemma} and \ref{lem:b cycles lemma}
each cycle $a_{i}$ and pair of cycles $b_{s}b_{s+1}$ has a strong
decomposition into at most three $\op$-elements. As each of the cycles
in the decomposition of $\sigma$ are disjoint, the $\op$-elements
commute where necessary to give $\sigma=\prod_{i=1}^{k_{a}}a_{i}\prod_{s=1}^{k_{b}}b_{s}\cdot g\in I_{p}(A_{n})^{3}$
as required. This leaves the case where $|\mu(g)|+c^{*}(g)\leq2p$:
firstly note that if $g=\sigma$ then the result follows immediately
from Lemma \ref{lem:first cd lemma}. If $g\neq\sigma$ and $|\mu(g)|<p$,
then there exists at least one sub-element $h$ of $\prod_{i=1}^{k_{a}}a_{i}\prod_{j=1}^{k_{c}}c_{j}\neq1$
to which Lemma \ref{lem:first cd lemma} Cases (1)-(5) applies. In
particular $gh\in I_{p}(A_{n})^{3}$ and this is a strong decomposition.
As $\left(\prod_{i=1}^{k_{a}}a_{i}\prod_{j=1}^{k_{c}}c_{j}\right)\backslash h$
has $p$-width at most 3 (and this is also a strong decomposition) by Lemmas
\ref{lem:a cycles lemma} and \ref{lem:b cycles lemma}, it follows
that $\pw(\sigma)\leq3$. Lastly, if $|\mu(g)|\geq p$ then the result
follows by Lemmas \ref{lem:a cycles lemma}, \ref{lem:b cycles lemma}
and \ref{lem:first cd lemma} in an identical fashion. This completes
the proof of Theorem \ref{thm:my theorem}.
\end{proof}

\appendix

\section{Alternative proof} \label{sec:Alt proof}
In this appendix we address the introductory remark concerning the work of Dvir \cite{Dvir}, and its application to the $p$-width of alternating groups. We were unaware of Theorem \ref{thm: Dvir Thm} at the time of writing, but as is evident below, it greatly simplifies the proof of Theorem \ref{thm:my theorem}. We hope that our earlier methods will still be of use to the reader. 

	\begin{defn}
	Let $D\in A_{n}$ and define
	\[ [D]:= \text{Conjugacy class of } D; \]	
	\[	r(D):=|\mu(D)|-c^{*}(D). \]
	
\end{defn}

\begin{thm} \cite[Thm. 10.2]{Dvir} \label{thm: Dvir Thm}
	Let $D\in A_{n} $ such that $D$ is not of cycle type $2^{k}$ for $k>1$. If $r(D)\geq \frac{1}{2}(n-1)$, then $[D]^3=A_{n}$. 
\end{thm}

\begin{cor} \label{cor:Dvir cor}
	Fix a prime $p\geq 3$ and suppose that $n\geq kp$ for some $k\in \mathbb{N}_{\geq 1}$. Let $D \in A_{n}$ have cycle type $(p^k,1^{n-kp})$. Suppose that $[D]^{3}\neq A_{n}$, then $n\geq (k+1)p$.
	
\end{cor}

\begin{proof}
	Firstly, if $k=1$ then (by Theorem \ref{thm: Dvir Thm}) the assumption that $[D]^{3}\neq A_{n}$ yields $(n-1)/2>p-1$ which is true if and only if $n>2p-1$.
	Now let $k \geq 2$ and suppose for a contradiction that $n<(k+1)p$. Now by assumption (and therefore by Theorem \ref{thm: Dvir Thm}) $(n-1)/2>k(p-1)$. Hence \[
	\frac{(k+1)p-1}{2}>k(p-1). \]
	This then reduces to $2k-1>p(k-1)\geq 3k-3$  as $p\geq 3$. But this is a contradiction.
\end{proof}

\textbf{Alternative proof of Theorem \ref{thm:my theorem} }
	Firstly, if $p=2$ then the result is given by \cite[2.3]{AJM 2017}. Hence we may assume that $p$ is an odd prime. The result is then immediate from Corollary \ref{cor:Dvir cor} as $n$ is finite. \qedsymbol

\end{document}